\documentclass{amsart}[12pt]

\usepackage{amsfonts,amssymb,amsmath,amscd,amsthm, mathrsfs}
\usepackage{graphicx}
\usepackage{epstopdf}
\usepackage{hyperref}

\newtheorem{thm}{Theorem}[section]
\newtheorem{dfi}{Definition}[section]
\newtheorem{prop}{Proposition}[section]
\newtheorem{lm}{Lemma}[section]
\newtheorem{cor}{Corollary}[section]
\newtheorem{rem}{Remark}[section]
\newtheorem{qu}{Question}

\newcommand{\mc}[1]{\mathcal{#1}}
\newcommand{\rr}{\mathbb{R}}
\newcommand{\zz}{\mathbb{Z}}
\newcommand{\cc}{\mathbb{C}}

\newcommand{\al}{\alpha}

\newcommand{\dl}{\delta}

\newcommand{\ep}{\varepsilon}

\newcommand{\gm}{\gamma}

\newcommand{\lmd}{\lambda}

\newcommand{\Om}{\Omega}
\newcommand{\sg}{\sigma}

\newcommand{\tht}{\theta}
\newcommand{\kp}{\kappa}

\newcommand{\ey}{\frac{1}{2}}

\newcommand{\pinf}{+\infty}

\newcommand{\cx}{\mc{X}}
\newcommand{\mcc}{\mc{C}}
\newcommand{\kc}{K_{\mcc}}
\newcommand{\cs}{\mc{S}}

\newcommand{\bw}{\mathbf{w}}
\newcommand{\bn}{\mathbf{n}}
\newcommand{\bk}{\mathbf{k}}
\newcommand{\ik}{I_{\mathbf{k}}}
\newcommand{\lk}{L_{\mathbf{k}}}
\newcommand{\uk}{U_{\mathbf{k}}}
\newcommand{\kk}{K_{\mathbf{k}}}

\newcommand{\sbb}{\bar{s}}
\newcommand{\qb}{\bar{q}}
\newcommand{\qe}{\bar{q}^{\ep}}
\newcommand{\pw}{\frac{2}{2+\al}}
\newcommand{\pwr}{\frac{2+\al}{2}}

\newcommand{\xl}{x^{\lmd}}
\newcommand{\xln}{x^{\lmd_n}}
\newcommand{\ql}{q^{\lmd}}
\newcommand{\qln}{q^{\lmd_n}}

\newcommand{\yln}{y^{\lmd_n}}
\newcommand{\dql}{\frac{d}{dt}q^{\lmd_n}}

\newcommand{\tp}{\tht^+}

\newcommand{\lb}{\bar{L}}
\newcommand{\gb}{\bar{\gm}}

\newcommand{\xe}{x^{\ep}}
\newcommand{\xh}{\hat{\mc{X}}}

\newcommand{\tb}{\bar{T}}
\newcommand{\omh}{\hat{\Omega}}

\newcommand{\dx}{\dot{x}}
\newcommand{\xd}{\dot{x}}
\newcommand{\qd}{\dot{q}}

\newcommand{\jk}{J_{\bk}}
\newcommand{\td}{\dot{\tht}}

\newcommand{\rb}{\bar{\rho}}

\begin{document}

	\title[Shape Space Figure-$8$ Solution]{Shape Space Figure-$8$ Solution of Three Body Problem with Two Equal Masses}
	\author{Guowei Yu}
	\email{yu@math.utoronto.ca}
	\date{04/11/2016}
	
	\address{Department of Mathematics, University of Toronto}
	
	\begin{abstract} 
		In a preprint by Montgomery \cite{Mo99}, the author attempted to prove the existence of a shape space Figure-$8$ solution of the Newtonian three body problem with two equal masses (it looks like a figure $8$ in the shape space, which is different from the famous Figure-$8$ solution with three equal masses \cite{CM00}). Unfortunately there is an error in the proof and the problem is still open.

		Consider the $\alpha$-homogeneous Newton-type potential, $1/r^{\alpha},$ using action minimization method, we prove the existence of this solution, for $\alpha \in (1,2)$; for $\alpha=1$ (the Newtonian potential), an extra condition is required, which unfortunately seems hard to verify at this moment.
	\end{abstract}
	
	\maketitle
	
\section{Introduction} \label{sec:intro}

The motion of three point masses $m_j>0$, $j=1,2,3$, in a plane under $\al$-homogeneous Newtonian-like potential is described by the following equation

\begin{equation} \label{three body}
m_j \ddot{x}_j = \frac{\partial}{\partial x_j} U(x), \quad j = 1,2,3,
\end{equation}
where $x=(x_1, x_2, x_3) \in \cc^3$ with $x_j \in \cc$ representing the position of $m_j,$ and $U(x)$ is the (negative) potential energy 
\begin{equation}
\label{potential} U(x) = \sum_{1 \le j < k \le 3} \frac{m_j m_k} { \al |x_j - x_k|^{\al}}.
\end{equation}

In the literature, the potential is usually referred as a \emph{strong force} when $\al \ge 2$ and a \emph{weak force} when $0 < \al < 2$. The Newtonian potential corresponds to $\al =1$. In this paper we only consider $1 \le \al <2$ and to emphasize the specialty of the Newtonian potential, the term \emph{weak force} will only refer to cases $1< \al <2$.

Equation \eqref{three body} is the Euler-Lagrange equation of the action functional 
\begin{equation}
\label{action} A_{L} (x, [T_1, T_2]) = \int_{T_1}^{T_2} L(x, \dot{x}) \,dt, \;  x \in H^1([T_1,T_2], \cc^3),
\end{equation}
where $ L(x, \dot{x}) = K(\dot{x}) + U(x)$ is the Lagrange and $ K( \dot{x}) = \frac{1}{2} \sum_{j =1}^{3} m_j | \dot{x}_j|^2$ is the kinetic energy. For simplicity, set $A_L(x, T) = A_L(x, [0, T]),$ for any $T>0$.

As the three body problem is invariant under linear translation, the center of mass can be fixed at the origin, and we set
$$ \cx = \{ x \in \cc^3: \sum_{j=1}^3 m_j x_j =0 \}.$$
The subset of collision-free configurations will be denoted by $\xh = \{ x \in \cx: x_j \ne x_k,\; \forall 1 \le j < k \le 3 \}$.

In the celebrated paper by Chenciner and Montgomery \cite{CM00}, the famous Figure-8 solution (three equal masses chase each other on a fixed loop in a plane with the shape of figure $8$) was proved. It belongs to a special class of solutions now known as simple choreographies, see \cite{CGMS02}, \cite{Y15c}.

An interesting story was told in the appendix of \cite{CM00} regarding the origin of the paper: A preprint titled ``Figure $8$s with Three Bodies'' by Montgomery \cite{Mo99}, was submitted to \emph{Nonlinearity} and Chenciner was asked to be the referee. In the preprint, Montgomery attempted to prove the existence of two zero angular momentum \emph{reduced periodic} solutions (periodic after modulo a proper rotation) of the Newtonian three body problem. One under the condition that all three masses are equal and the other under the condition that two masses are equal to each other. While the first solution eventually leads to the proof of the Figure-$8$ solution in \cite{CM00}, the proof of the second solution was found to be an error. 

What's interesting is that the `Figure $8$s' in the title of the preprint was referring to the second solution, because once we established its existence, it should have the following features: if one puts the system in a moving frame such that the line between the two equal masses is fixed and further apply a time-dependent homothety to the motion so that they are fixed points on the lines all the time, then the third mass moves along a curve having the topological type of a figure $8$ with each circle of the $8$ surrounding one of the two equal masses. This means after projecting the solution to the shape space, it is a figure $8$ surrounding two of the binary collision rays, see Section \ref{shape}. Because of this, this solution will be called the \emph{shape space Figure-8 solution} throughout the paper. 

The approach proposed by Montgomery in \cite{Mo99} was to find the desired solution as a minimizer of the action functional under certain symmetric and topological constraints. The main difficulty is to show the minimizer is collision-free. In the last fifteen years, lots of progresses have been made to overcome this. We briefly summarize in the following.  
\begin{itemize}
    \item \textbf{Local deformation}: based on asymptotic analysis near an isolated collision and the `` blow-up '' technique, one tries to show after a local deformation of the collision path near the isolated collision, there is a new path with strictly smaller action value. For the details see \cite{Mo98}, \cite{Ve02}, \cite{Ce02} and \cite{FT04}. 

	\item \textbf{Level estimate}: one gives a sharp lower bound estimate of the action functional among all the collision paths in the set of admissible paths and then try to find a collision-free test path within the set of admissible paths, whose action value is strictly smaller than the previous  lower bound estimate. See \cite{CM00}, \cite{Ch03a}, \cite{Ch08} and \cite{Ch13}.	
\end{itemize}
Despite of all the progresses, a proof of the shape space Figure-$8$ solution is still missing for any $1 \le \al<2$. 

In this paper, based on some new local deformation result near an isolated binary collision (see Section \ref{bicoll}), we show the approach laid out by Montgomery in \cite{Mo99} can be carried out and a shape space Figure-$8$ solution exists, for weak force potentials ($1<\al<2$); for the Newtonian potential ($\al=1$), an extra condition \eqref{assume} is required. To be precise, we have the following results. 
\begin{thm} \label{thm 1}
     Assume $m_1 =m_2$. When $1 < \al < 2$, for any $\tb >0$,  there is a zero angular momentum periodic solution $x \in C^2(\rr / \bar{T}\zz, \xh)$ of equation \eqref{three body} satisfying: 
     \begin{enumerate}
      \item[(a).] when $t \in \{0, \tb/2 \}$, $x(t) \in \rr^3 \subset \cc^3$ satisfying
      $$ x_1(0) < x_2(0) < x_3(0), \quad  x_3(\tb/2) < x_1(\tb/2) < x_2 (\tb/2);$$
      \item[(b).] when $t \in \{\tb/4, 3\tb /4 \}$, $x(t)$ is an Euler configuration with $x_3(t) =0$ and $ x_1(t) = -x_2(t);$ 
      \item[(c).] when $t \in [0,\tb)$, $x(t)$ experiences exactly four collinear configurations at the moments $t =0, \tb/4, \tb/2,$ and $3\tb/4;$
      \item[(d).] for any $t \in \rr/\bar{T}\zz$,
      $$ (x_1, x_2, x_3)(t) = (-x_2, -x_1, -x_3)(\tb/2 -t);$$
      $$ (x_1, x_2, x_3)(t) = (\bar{x}_1, \bar{x}_2, \bar{x}_3)(\tb -t), $$
      where $\bar{x}_j$ is the complex conjugate of $x_j$.
     \end{enumerate}

     When $\al =1$, the above results hold if the following condition is satisfied
     \begin{equation}
 	\label{assume} A_L(x^*, \tb/4) > \inf \{ A_L(y, \tb/4): y \in \Om \},
 	\end{equation} 
 	where $x^*$ is a Schubart solution $($it is a collinear collision solution defined in Definition \ref{Schubart}$)$ and $\Om$ is the set of admissible paths defined as in Definition \ref{dfn:Om}. 
\end{thm}

Only \textbf{local deformation} are used throughout the paper to rule out collisions and in some sense Theorem \ref{thm 1} is the best result, one could expect based on this method, see Remark \ref{rem: Gordon}. It will be really interesting, if someone can verify or disprove condition \eqref{assume} for certain choices of masses. At this moment, this seems hard to do. In general, \textbf{level estimate} used in \cite{Ch08} and \cite{Ch13} can give a good lower bound estimate of the action value of a collision path. However it is not the case when the collision path is collinear (all three masses stay on a single line all the time), which is exactly the case for a Schubart solution. 


\section{Geometry of the Shape Space} \label{shape}

A proof of Theorem \ref{thm 1} will be given in this section. For this, first we briefly recall the shape space and the syzygy sequences of the planar three body problem. The details can be found in a series of papers by Montgomery: \cite{Mo96}, \cite{Mo98}, \cite{Mo99}, \cite{Mo02}, as well as the beautiful expository article \cite{Mo15} by him. 

The configuration space $\cx$ has four (real) dimension. After modulo rotations given by $SO(2)$ action: $ e^{i \tht}(x_1, x_2, x_3) = (e^{i \tht}x_1, e^{i \tht} x_2, e^{i \tht} x_3),$ the quotient space $\mcc = \cx / SO(2)$, representing the congruence classes of triangles formed by the three masses, is homeomorphic to $\rr^3$ and will be referred as \emph{the shape space}. 

The above process can be realized by a map 
$$ \pi: \cx \to \mcc \approx \rr^3; \; \pi(x) = \bw =(w_1, w_2, w_3).$$
Hence $w_3$ represents the signed area of the corresponding triangle in the configuration space multiplied by a constant depending on the masses, so $\pi$ maps the collinear configurations (all three masses lie on a single straight line) to the plane $\{ w_3 = 0 \}$. In astronomy, a collinear configuration is also called a \emph{syzygy}, see \cite{Mo02}, \cite{MMV12}. Depending on which mass is in the middle, we call it a type-$j$ syzygy. The plane $\{ w_3 = 0\} \subset \mcc$ will be called \emph{the syzygy plane}. 

\begin{figure}
	\centering
	\includegraphics[scale=0.6]{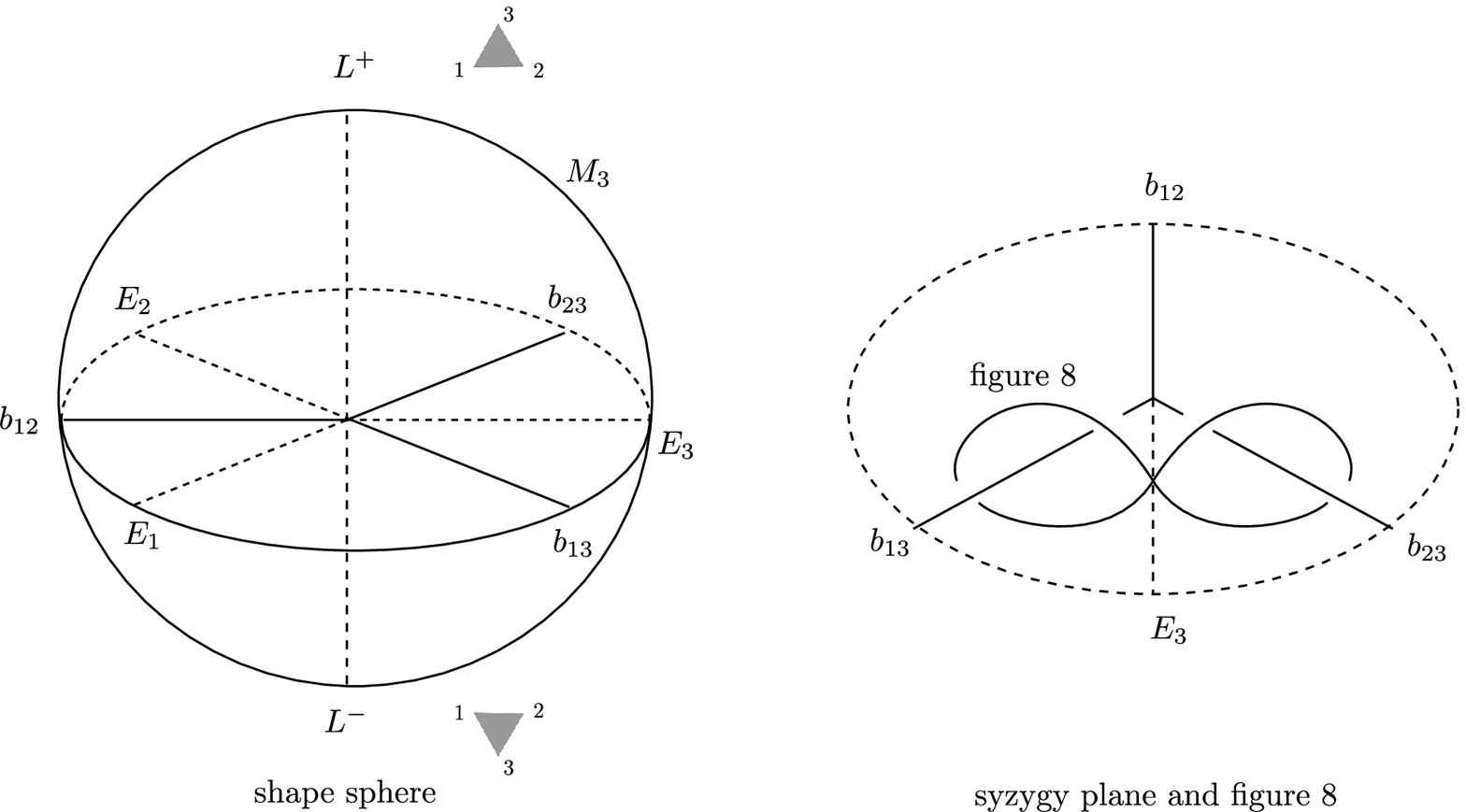}
	\caption{}
	\label{8-1}
\end{figure}

Let $I(x)= \sum_{j=1}^3 m_j |x_j|^2$ be the moment of inertia. The two dimensional sphere $ \cs = \{ \bw = \pi(x): I(x)= 1 \} \subset \mcc$ represents all oriented similarity classes of triangles and will be referred as \emph{the shape sphere}. 

Although $\mcc$ is homeomorphic to $\rr^3$, as a metric space it is not isometric to $\rr^3$, but to the cone $c(\cs)$. In general, the cone $c(M)$ over a Riemannian manifold $M$ with metric $ds^2$ is a metric space with one higher dimension. Topologically it is homeomorphic to $M \times [0, \pinf) / M \times \{0\}$ with $M \times \{ 0\}$ collapse to a single point. Given a $\bw \in \cs$ (or $W \subset \cs$), we use $c(\bw)$ (or $c(W)$) to represent the cone over $\bw$ (or $W$). 

There are five different normalized central configurations (modulo rotations): two of them are the equilateral Lagrangian configurations $L^{\pm}$ with different orientations; the other three are the collinear Euler configurations $E_j$, $j=1,2,3$, with $m_j$ in the middle. In the shape sphere, $L^{\pm}$ correspond to the north and south poles and $E_j$'s are located on the equator $\cs \cap \{ w_3 = 0 \}.$ See the left picture in Figure \ref{8-1}.

The three $E_j$'s divide the equator into three disjoint arcs. Each arc contains a binary collision $b_{jk}, 1 \le j<k \le3$ (the sub-index means the collision is between $m_j$ and $m_k$). $c(b_{jk})$'s represent the three binary collision rays emanating from the origin. After deleting them from the shape space, $ \mcc^* = \mcc \setminus \{ c(b_{jk}): 1 \le j< k \le 3 \}$ has nontrivial fundamental group. As $\mcc^*$ is homotopic to $\cs \setminus \{ b_{jk}: 1 \le j < k \le 3 \}$, its fundamental group is isomorphic to the projective colored braids group, see \cite{Mo96}. 

Given any two points in the shape space, they represent two oriented congruence classes of triangles in $\cx$. A minimizer of the action functional $A_L$ among all paths connecting those two oriented congruence classes must have zero angular momentum. As a result, its projection on the space can seen as an action minimizer of the action functional associated with the Lagrange $L_{\mcc}$: 
$$ L_{\mcc} = \kc + U, \;\; \kc = K - |J|^2/(2I),$$ 
where $U$, $K$, $J$ and $I$ are correspondingly the potential energy, the kinetic energy, the angular momentum and the moment of inertia. This establishes a connection between the variational problem in the original configuration space and the one in the shape space.  

Inspired by the basic theorem that every nontrivial free homotopy class of a compact Riemannian manifold can be realized by a minimal closed geodesic, Montgomery asked the following questions in \cite{Mo98} (similar questions were also asked by Wu-Yi Hsiang):

\begin{qu}
	\label{qu1} Can each free homotopy class of $\mcc^*$ be realized by a \textbf{zero angular momentum} reduced periodic solution of the Newtonian three body problem?
\end{qu}

\begin{qu}
	\label{qu2} The same question as above but without the condition of angular momentum being zero. 
\end{qu}

A nice way to represent the free homotopy classes is using the syzygy sequences. Given a free homotopy class of $\mcc^*$, a generic loop contained in this class has a discrete set of syzygies. Write down the syzygy types, namely $1,2$ or $3$, experienced by the loop following their temporal order in a single period, we got a corresponding \emph{syzygy sequence} of the loop. Such a sequence may have two or more of the same letters from $\{1,2,3\}$ appearing consecutively in a row. Such a phenomena will be called \emph{stutter}. After canceling all the stutters, the remaining sequence will be called the \emph{reduced syzygy sequence} of the loop. 

If two generic loops are from two different free homotopy classes, they have different reduced syzygy sequences; if they are from the same free homotopy class, then they have the same reduced syzygy sequence (one may need to reverse the time on one of the loop). As a result, there is a $2$-to-$1$ map from the periodic reduced syzygy sequences to the free homotopy classes of $\mcc^*$, except the empty sequence, which is the only pre-image of the trivial free homotopy class. 

Now Question \ref{qu1} and \ref{qu2} can be rephrased with \emph{free homotopy class} replaced by \emph{reduced syzygy sequence}. 

In a recent paper \cite{MM14} by Moeckel and Montgomery, Question \ref{qu2} was answered affirmatively using non-variational method. However their proof requires the angular momentum to be small but non-zero, which leaves Question \ref{qu1} still open. Meanwhile zero-angular momentum solutions can be found as action minimizations of the action functional define by $L_{\mcc}$. If we replace the Newtonian potential by a strong force potential, then Question \ref{qu1} was answered completely by Montgomery in \cite{Mo98} using variational approach, as in this case a minimizer can not have any collision. However when the potential is Newtonian or a weak force, for most free homotopy classes, an action minimizer is likely to contain collision, see \cite{Mo02a}.

Despite of this, for some free homotopy classes or reduced syzygy sequences, the variational approach may still work, if we can impose additional constraints on the paths. As it is well-known now, this can be done when some of the masses are equal. For example the Figure-$8$ solution of three equal masses is an action minimizer under certain symmetric constraints, and it is a zero angular momentum solution realize the periodic reduced syzygy sequence $123123$. 

Besides the above example, when $m_1 =m_2$, $2313$ may also be realizable by a zero angular momentum reduced periodic solution as suggested by Montgomery in \cite{Mo02a}. In the following, the readers will see the zero angular momentum periodic solution obtained in Theorem \ref{thm 1} does indeed realize this syzygy sequence. Obviously corresponding results hold for $3212$ (or $2131$), when $m_1 =m_3$ (or $m_2 = m_3$). 

Following the approach given in \cite{Mo99}, we first formulate the corresponding action minimization problem. Let $T_0=\tb/4$ for the rest of the paper.
\begin{dfi}
 \label{dfn:Om} We define $ \omh$ as the subspace of $H^1([0,T_0], \xh)$ consisting of all paths, starting at a type-$2$ syzygy with all the three masses lying on the real axis in the order: $x_1 < x_2 < x_3$ of arbitrary size and ending at the Euler configuration $E_3$ of arbitrary size and arbitrary angle. The weak closure of $\omh$ in $H^1([0,T_0], \cx)$ will be denoted by  $\Om$. 
\end{dfi}
After projecting to the shape space $\pi(\Om) =\{ \pi(x): x \in \Om \}$ contains all the paths starting at the closed sector in the syzygy plane between $c(b_{12})$ and $c(b_{23})$, and ending at $c(E_3)$.

\begin{prop} \label{prop}
	Assume $m_1=m_2$. When $1< \al <2$, the infimum of the action functional $A_L$ in $\Om$ is a minimum and any action minimizer $x \in \Om$ must be collision-free.
\end{prop}

\begin{prop}
 \label{prop 2} Assume $m_1=m_2$. When $\al =1$, the infimum of the action functional $A_L$ in $\Om$ is a minimum and any action minimizer $x \in \Om$ must satisfies one of the followings: 
 \begin{enumerate}
 \item $x$ is collision-free; 
 \item $x \in \Om \cap H^1([0,T_0], \rr^3)$ contains a single binary collision at $t=0$ $(x_1(0) < x_2(0) =x_3(0))$, and it is a quarter of a Schubart solution. 
 \end{enumerate}
\end{prop}

\begin{dfi}
	\label{Schubart} When $m_1=m_2$, a collision solution $x^*: \rr / \bar{T} \zz \to \cx$ of \eqref{three body} is called a Schubart solution, if $x(t) \in \rr^3$, for any $t$, and the following conditions are satisfied:
	\begin{enumerate}
	\item $x^*$ satisfies condition (d) in Theorem \ref{thm 1}; 				
	\item $x^*_1(0) <0 < x^*_3(0) = x^*_2(0)$ and $\dot{x}^*(0) =0$;
	\item for $t \in (0, T_0)$, $x^*_3(t)$ is strictly decreasing, $x^*_j(t)$, $j=1,2$, is strictly increasing and $x_1(t) < x_3(t) < x_2(t)$;
	\item $x^*_3(T_0) =0$ and $x^*_2(T_0) = -x^*_1(T_0).$
	\end{enumerate}
\end{dfi}

This solution was discovered numerically by Schubart \cite{Sc56} in the case of Newtonian potential and the following was proved by Venturelli (\cite{Ve08})
\begin{prop}
\label{Venturelli} When $\al=1$ and $m_1=m_2$, the infimum of the action functional $A_L$ in $\Om \cap H^1([0,T_0], \rr^3)$ is a minimum and an action minimizer is a quarter of a Schubart solution. \end{prop}
By this result, when $m_1=m_2$, if a minimizer $x \in \Om$ obtained in Proposition \ref{prop 2} is not collision-free, it must be a quarter of a Schubart solution. 

The proofs of Proposition \ref{prop} and \ref{prop 2} will be given in Section \ref{weak} and \ref{sec newton}. Right now we give a proof of Theorem \ref{thm 1} based on them. 
\begin{proof} 

 [\textbf{Theorem \ref{thm 1}}]
 First let's assume $1<\al<2$. By Proposition \ref{prop}, there is an $x \in C^2([0,T_0], \cx)$, which is a collision-free minimizer of $A_L$ in $\Om$. Therefore $x$ is a solution of \eqref{three body} with zero angular momentum. Due to the symmetry with respect to the syzygy plane, we can assume $\pi(x(t))$ always lies above or on the syzygy plane. 
 
 Notice that $\pi(x(t))$ only intersects the syzygy plane at $t=0$ and $ T_0.$ First it can not stay on the syzygy plane all the time, because then it must experience a collision. Now if it has an intersection with the syzygy plane at a time other than $0$ or $T_0$, then $\pi(x(t))$ must has a corner at some of those intersecting moments and it cannot be a minimizer. 
 
 We can continue the solution $x|_{[0,T_0]}$ to the time interval $[T_0, 2T_0]$ by twisting it through the Euler configuration $E_3$ and reversing the time: 
 $$ x(T_0+t) = H_3(x(T_0-t)), \quad \forall t \in [0, T_0],$$
 where the twist $H_3$ is the composition of reflection about the syzygy plane with reflection about the plane $c(M_3)$ ($M_3=\{\pi(x): |x_3-x_1| = |x_3-x_2|\} \cap \cs$ is a meridian on the space sphere, see the left picture in Figure \ref{8-1}). It is a symmetry of the action when $m_1 = m_2.$
 
 In the inertial plane, it is realized by the following equation
 $$ (x_1, x_2, x_3) (T_0 + t) = (-x_2, -x_1, -x_3)(T_0-t), \quad \forall t \in [0, T_0]. $$
 
 Now $x|_{[0, 2T_0]}$ is a path starting from a type-$2$ syzygy, ending at a type-$1$ syzygy and passing $E_3$ at the half time. By the minimizing and symmetric properties of $x$, it must hit the syzygy plane orthogonally on both ends. Hence we can continue the solution by reflecting it with respect to the syzygy plane and reversing the time. This gives us a closed loop in the shape of figure $8$ in the shape space as show in the right picture of Figure \ref{8-1}. Therefore it is also a reduced periodic solution with period $\tb =4T_0$, which realizes the syzygy sequence $2313.$
 
 To recover the motion in the inertial plane and show that it is in fact a periodic solution not just reduced periodic. We use the area rule obtained in \cite{Mo96}. Following the same argument used in \cite{CM00}, the angle between the straight lines passing all the three bodies at $t=0$ and $t=2T_0$, must be zero. Since we assume all masses start from the real axis, they all come back to it at the time $2T_0$. As a result the reflection with respect to the syzygy plane we performed on $x|_{[0,2T_0]}$ is just a reflection respect to the real axis and is realized by the following equation in the inertial plane
 $$ (x_1, x_2, x_3)(2T_0 +t) = (\bar{x}_1, \bar{x}_2, \bar{x}_3) (2T_0-t), \quad \forall t \in [0, 2T_0].$$
 This implies $x(t)$ is a periodic solution in the inertial plane with all the required properties.

 When $\al=1$, By Proposition \ref{prop 2}, there is an $x \in \Om$, which is a minimizer of the action functional $A_L$ in $\Om$. If condition \eqref{assume} holds, then by Proposition \ref{prop 2} and \ref{Venturelli}, $x$ must be collision-free. The rest follows from the same argument given as above. 
\end{proof}

\section{Local Deformation} \label{sec local}
We prove several local deformation results in this section that will be needed. Results given in this section hold for the $N$-body problem in $\rr^d$, for any $N, d \ge 2$, 
\begin{equation}
\label{N body} m_j \ddot{x}_j = \frac{\partial}{\partial x_j} U(x), \quad j =1, \dots, N,
\end{equation}
with the corresponding Lagrange $L(x, \dx) = K(\dx) + U(x)$,
$$K(\dx) = \ey \sum_{j=1}^N m_j |\dx_j|^2, \;\; U(x) = \sum_{1 \le j < k \le N} \frac{m_j m_k}{\al |x_j - x_k|^{\al}}.$$

Start with an assumption that an action minimizer contains an isolated collision, there are two different approaches of \textbf{local deformation}: one of them gets a contradiction by showing the average action value of locally deformed paths along a large enough set of directions is smaller than the action value of the collision minimizer (this was first introduce by Marchal \cite{Ce02}, then generalized by Ferrario and Terracini \cite{FT04}); the other gets a contradiction by showing directly the action value of some locally deformed path along some properly chosen direction is smaller than the action value of the collision minimizer (to our knowledge, this method first appeared in \cite{Mo99}, and then \cite{Ve02}). 

In general the first approach is more powerful. In particular when there is only symmetric constraints, it only requires the \emph{rotating circle property} to be satisfied, see \cite{FT04}. Although it does not work when topological constraints are also imposed. This makes it not applicable to the shape space Figure-8 solution, as the corresponding symmetric constraints do not satisfy the rotating circle property and topological constraints are involved in some sense (although it makes boundary constraints after symmetries). While the second approach is not as powerful as the first in general, it can be applied to problem with certain symmetric and topological constraints

In the following, we prove some local deformation results based on the second approach. The results essentially the same as those given in \cite{Mo99} and \cite{Ve02}. We include them here for two reasons: first in the previous references the results are only for the Newtonian potential; second this gives us the opportunity to introduce some notations and technical results that will be needed in the next section.

Let $\bn:=\{1, \dots, N \}$, for any subset of indices $\bk \subset \bn$, the sub-system of $\bk$-body problem has its Lagrange $ L_{\bk}(x, \dot{x})= K_{\bk}(\dot{x}) + U_{\bk}(x),$ where
$$ K_{\bk}(\dot{x}) = \ey \sum_{j \in \bk} m_j |\dot{x}_j|^2, \quad U_{\bk}(x) = \sum_{ j, k \in \bk; j<k} \frac{m_j m_k}{\al |x_j -x_k|^{\al}}. $$

Any $y = (y_j)_{j \in \bk}$ with $y_j \in \rr^d$ will be called a \emph{$\bk$-configuration}, and a \emph{centered $\bk$-configuration}, if $\sum_{j \in \bk} m_j y_j =0.$ The set of all centered $\bk$-configurations will denoted by 
$$ \cx^{\bk} := \{ y = ( y_j)_{j \in \bk}: \sum_{j \in \bk} m_j y_j =0 \}. $$
In this section, by a configuration we will only mean an $\bn$-configuration. Given a configuration $x$ and a $\bk$-configuration $y$, $z= x+y$ is defined as
$$ z_j = x_j + y_j, \text{ if } j \in \bk; \quad z_j = x_j, \text{ if } j \in \bn \setminus \bk. $$

\begin{dfi}
 \label{cluster collision} Given an $x \in H^1([0, \dl], \cx)$ with $\dl >0$, we say $x(0)$ is an \textbf{isolated $\bk$-cluster collision}, if $x(t), t \in (0, \dl]$ is collision-free and $x(0)$ has a $\bk$-cluster collision, i.e.,
$$ x_k(0) = x_0(0), \; \forall k \in \bk; \quad x_j(0) \ne x_0(0), \; \forall j \in \bn \setminus \bk,$$
where $x_0(t)$ is the center of mass of the $\bk$-body,
\begin{equation}
\label{x0} x_0(t) = \frac{\sum_{k \in \bk} m_k x_k(t)}{m_0}, \quad m_0 := \sum_{k \in \bk} m_k.
\end{equation} 

Such an $x$ will be called an \textbf{isolated $\bk$-cluster collision solution}, if $x(t), t \in (0, \dl],$ satisfies equation \eqref{N body}.
\end{dfi}

Fix an isolated $\bk$-cluster collision solution $x \in H^1([0, \dl], \cx)$ for the rest of this section. Let $\ik$ be the moment of inertia of the $\bk$-body with respect to their center of mass $x_0$, and $q$ the relative positions of $m_k$, $k \in \ik$, with respect to $x_0$,
\begin{equation} \label{ik}
 \ik = \ik (x) =\sum_{k \in \bk} m_k |x_k - x_0|^2,
\end{equation}
\begin{equation} \label{q}
 q = (q_k)_{k \in \bk}= (x_k - x_0)_{i \in \bk}.
\end{equation}

While $\ik$ represents the size of the $\bk$ cluster, its shape can be described by the normalized centered $\bk$-configuration 
\begin{equation} \label{s}
 s= (s_k)_{k \in \bk} = \big(\ik^{-\ey} (x_k - x_0) \big)_{k \in \bk} = \big(\ik^{-\ey} q_k \big)_{k \in \bk}.
\end{equation}
Obviously $\ik(s) =1.$

Choose a sequence of times $\{t_n >0\}$, such that $\lim_{n \to \infty} t_n=0$ and $\{ s(t_n) \}$ converges to a normalized centered $\bk$-configuration $\sbb=(\sbb_k)_{k \in \bk} $ (such a sequence always exists as the set of normalized centered $\bk$-configurations is compact), then the following is a well-known result (for a proof see \cite{FT04}).

\begin{prop} \label{cc}
	 $  \sbb = \lim_{n \to \pinf} s(t_n)$ is a central configuration of the $\bk$-body problem. 
\end{prop}

Let $\sg =(\sg_k)_{k \in \bk}$ be a normalized centered $\bk$-configuration, i.e.,
\begin{equation}
\label{eqn:sg}\ik(\sg) = \sum_{k \in \bk} m_k |\sg_k|^2=1; \quad \sum_{k \in \bk} m_k \sg_k = 0.
\end{equation}
The following proposition is the first local deformation result we obtain.
\begin{prop}
	\label{deform1} Let $\sg_{jk} = \sg_j - \sg_k$ and $\sbb_{jk} = \sbb_j - \sbb_k,$ if the following condition holds
	\begin{equation}
	 \label{eqn:dfm dir} \langle \sg_{jk}, \sbb_{jk} \rangle \ge 0, \; \forall \{j \ne k\} \subset \bk
	\end{equation}
        then for $\dl_0, \ep_0>0$ small enough, there is a new path $y \in H^1([0, \dl], \cx)$ satisfying $A_L(y, \dl) < A_L(x, \dl),$ and 
        \begin{enumerate}
	 	\item[(a).] $y(t) = x(t),\; \forall t \in [\dl_0, \dl];$
	 	\item[(b).] $y(t)$ is collision-free, for any $t \in (0, \dl];$
	 	\item[(c).] $y|_{[0,\dl_0]}$ is a small deformation of $x|_{[0,\dl_0]}$ with $y(0) = x(0) + \ep_0 \sg,$ in particular $y_{k_1}(0) \ne y_{k_2}(0)$, for any $\{k_1 \ne k_2 \} \subset \bk$ and $y_k(0) \ne y_j(0)$, for any $k \in \bk$ and $j \notin \bk$.
	 	\end{enumerate}
	\end{prop}
To get the above result, first we will prove a similar result for the homothetic-parabolic solution $\qb(t)$ related to $\sbb$ (its energy is zero), which is defined as following
\begin{equation} \label{qb}
\qb(t) = (\kp t)^{\pw} \sbb, \text{ for any } t \in [0, \pinf),
\end{equation}
where $\kp$ is a constant only depending $\al$, $\sbb$ and the masses.

\begin{lm}
	\label{deform2} For any $T>0$, if condition \eqref{eqn:dfm dir} in Proposition \ref{deform1} holds, then for $\ep>0$ small enough, there is a collision-free $\qe \in H^1([0, T], \cx^{\bk})$ with $ A_{\lk}(\qe, T) < A_{\lk}(\qb, T)$, where $\qe(t) = \qb(t) + \ep f(t) \sg$, for any $t \in [0,T]$, and
	\begin{equation*}
	f(t) = \begin{cases}
	1, \quad & \text{ if } t \in [0, \ep^{\pwr}], \\
	1 + \frac{1}{\ep}(\ep^{\pwr}-t), \quad & \text{ if } t \in [\ep^{\pwr}, \ep^{\pwr}+ \ep], \\
	0, \quad & \text{ if } t \in [\ep^{\pwr} + \ep, T]. 
	\end{cases}
	\end{equation*}	
\end{lm}	

\begin{proof}
	By the definition of $\qe$ and $f$,
	\begin{align*}
	 &A_{\lk}(\qe, T) - A_{\lk}(\qb, T) = \int_0^T \lk(\qe, \dot{\qb}^{\ep}) - \lk(\qb, \dot{\qb}) \\
	   & = \int_0^{\ep^{\pwr}} \uk (\qe) - \uk(\qb)  + \int_{\ep^{\pwr}}^{\ep^{\pwr}+ \ep} \uk(\qe) - \uk(\qb)  + \int_{\ep^{\pwr}}^{\ep^{\pwr}+\ep} \kk(\dot{\qb}^{\ep}) - \kk(\dot{\qb})\\
	   & := A_1 + A_2 +A_3
	\end{align*}
	We will estimate each $A_j$ in the following. For any $\{j< k\} \subset \bk$, let 
	$$\qe_{jk} = \qe_j - \qe_k, \quad \qb_{jk} = \qb_j - \qb_k, \quad a_{jk}= \kp^{\pw}|\sbb_{jk}|, \quad c_{jk} = \kp^{\pw} \langle \sbb_{jk}, \sg_{jk} \rangle. $$
	Notice that $c_{jk} \ge 0$ by condition \eqref{eqn:dfm dir}. 
	
	Since $f(t) \ge 0, \forall t \in [0,T]$, the following holds for any $\{j <k\} \subset \bk$
	\begin{align*}
	 |\qb_{jk}(t) + \ep f(t) \sg_{jk}|^{\al} & = (|\qb_{jk}(t)|^2 + 2 \ep f(t) t^{\pw} c_{jk} + \ep^2 f^2(t) |\sg_{jk}|^2 )^{\frac{\al}{2}}  \ge |\qb_{jk}(t)|^{\al}. 
	\end{align*}
	This immediately tells us      
	\begin{equation} \label{a2}
	  A_2 = \sum_{j, k \in \bk; j<k} \frac{m_j m_k}{\al} \int_{\ep^{\pwr}}^{\ep^{\pwr} + \ep} |\qb_{jk}(t) + \ep f(t) \sg_{jk}|^{-\al} - |\qb_{jk}(t)|^{-\al} \, dt \le 0.
	\end{equation}
	
	For $A_1$, we notice that when $0 \le t \le \ep^{\pwr}$,
	\begin{align*}
	 \uk(\qe(t)) & - \uk(\qb(t)) = \sum_{j, k \in \bk; j<k} \frac{m_j m_k}{\al |\qe_{jk}(t)|^{\al}} - \frac{m_j m_k}{\al | \qb_{jk}(t)|^{\al}} \\
	 & = \sum_{j, k \in \bk; j<k} \frac{m_j m_k}{\al} ( |\qb_{jk}(t) + \ep \sg_{jk}|^{-\al} - |\qb_{jk}(t)|^{-\al} ) \\
	 & = \sum_{j, k \in \bk; j<k} \frac{m_j m_k}{\al} [ (a_{jk}^2 t^{\frac{4}{2 +\al}} + 2 \ep c_{jk} t^{\frac{2}{2 + \al}} + \ep^2 |\sg_{jk}|^2]^{-\frac{\al}{2}}- (a_{jk}t^{\frac{2}{2 + \al}})^{-\al} ].
	\end{align*}
	Therefore 
	$$ A_1= \sum_{j, k \in \bk; j<k} \frac{m_j m_k}{\al } \int_0^{\ep^{\pwr}} (a_{jk}^2 t^{\frac{4}{2 +\al}} + 2 \ep c_{jk} t^{\frac{2}{2 + \al}} + \ep^2 |\sg_{jk}|^2]^{-\frac{\al}{2}}- (a_{jk}t^{\frac{2}{2 + \al}})^{-\al} \, dt.$$
	After a time reparameterization $ \tau = t^{\pw}/ \ep$,	
	\begin{align}
	 \frac{2\al}{2+ \al} A_1&  =  \sum_{j, k \in \bk; j<k} m_j m_k \int_0^1 \frac{\ep^{\pwr} \tau^{\frac{\al}{2}}}{ (a_{jk}^2 \ep^2 \tau^2 + 2 c_{jk} \ep^2 \tau + \ep^2 |\sg_{jk}|^2 )^{\frac{\al}{2}}}- \frac{\ep^{\pwr} \tau^{\frac{\al}{2}}}{(a_{jk} \ep \tau)^{\al}}  \, d \tau \nonumber \\
	 & = \ep^{\frac{2-\al}{2}}\sum_{j, k \in \bk; j<k} m_j m_k \int_0^1 \frac{ \tau^{\frac{\al}{2}} }{(a_{jk}^2 \tau^2 + 2 c_{jk} \tau + |\sg_{jk}|^2)^{\frac{\al}{2}}} -\frac{\tau^{\frac{\al}{2}}}{(a_{jk}\tau)^{\al}}\, d\tau \nonumber \\
	 & \le \ep^{\frac{2-\al}{2}}  \sum_{j, k \in \bk; j<k} m_j m_k \int_0^1 [ (a_{jk}^2 \tau^2 + |\sg_{jk}|^2)^{-\frac{\al}{2}} - (a^2_{jk}\tau^2)^{-\frac{\al}{2}}] \tau^{\frac{\al}{2}} \, d\tau. \nonumber	  
	\end{align}
	The last inequality follows from $c_{jk} \ge 0$. Because $\sg$ is a normalized centered $\bk$-configuration, there is at least one $\sg_{jk}$ with $|\sg_{jk}| >0$. As a result, there is a constant $C_1 >0$, such that 
	\begin{equation}
	 \label{a1} A_1 \le -C_1 \ep^{\frac{2-\al}{2}}.
	\end{equation}
	In the rest of the paper, $C$ and $C_j$, $j \in \zz$, always represent positive constants. Meanwhile by a straight forward computation,
\begin{equation*} 
\begin{split}
 A_3 & = \int_{\ep^{\pwr}}^{\ep^{\pwr} + \ep} \kk(\dot{\qb}^{\ep}) - \kk(\dot{\qb}) \,dt = \sum_{k \in \bk} \frac{m_k}{2} \int_{\ep^{\pwr}}^{\ep^{\pwr}+\ep} |\dot{\qb}_k(t) - \sg_k|^2 - |\dot{\qb}_k(t)|^2 \, dt \\
 & = \sum_{k \in \bk} \frac{m_k}{2} \int_{\ep^{\pwr}}^{\ep^{\pwr}+\ep} | \sg_k|^2 - \frac{4}{2 +\al} \kp^{\pw} t^{-\frac{\al}{2 + \al}} \langle \sg_k, \sbb_k \rangle \, dt\\
 & = \frac{\ep}{2} \sum_{ k \in \bk} m_k |\sg_k|^2 -\frac{2}{2+\al} \kp^{\pw} \langle \sg, \sbb \rangle_m \int_{\ep^{\pwr}}^{\ep^{\pwr}+\ep} t^{-\frac{\al}{2+ \al}} \,dt. \\
 \end{split}
\end{equation*}
Notice that condition \eqref{eqn:dfm dir} implies $\langle \sg, \sbb \rangle_m \ge 0$. Then by $\sum_{k \in \bk} m_k |\sg_k|^2 =1$, we get 
       \begin{equation}
        A_3 \le \ep /2.
       \end{equation}
       
Following the above estimates
       $$ A_1 + A_2 + A_3 \le \ep/2 - C_1 \ep^{\frac{2-\al}{2}} <0,$$
       for $\ep>0$ small enough, as $0 < \frac{2-\al}{2} < 1$ for any $\al \in [1,2)$. This finishes our proof. 
       	
\end{proof}

To get a proof of Proposition \ref{deform1} using the above result, we need the \emph{blow-up} technique introduced by Terracini, see \cite{FT04}. 
\begin{dfi}
Given a $\lmd > 0,$ we say $\xl: [0, \dl/\lmd] \to \cx$ is a \textbf{$\lmd$-blow up} of $x$, where 
	$$ \xl(t) = \big(\xl_j(t)\big)_{j \in \bn} = \big( \lmd^{-\pw} x_j(\lmd t ) \big)_{j \in \bn}, \;\; \forall t \in [0, \dl/\lmd] $$
and $\ql: [0, \dl/\lmd] \to \cx^{\bk} $ is a \textbf{$\lmd$-blow up} of $q$, where
	$$ \ql(t) = \big(\ql_k(t)\big)_{j \in \bk} = \big( \lmd^{-\pw} q_k(\lmd t) \big)_{k \in \bn}, \;\; \forall t \in [0, \dl/\lmd].$$
\end{dfi}

Given a sequence $\{\lmd_n >0\}$ with $\lim_{n \to \infty} \lmd_n =0$, let $x^{\lmd_n}$ and $q^{\lmd_n}$ be the corresponding $\lmd_n$-blow up of $x$ and $q$. The following result was proved in \cite[(7.4)]{FT04}. 

\begin{prop} \label{uniform converges}
 If $\{ s(\lmd_n) \in \cx^{\bk} \}$ converges to $\sbb$, then for any $T>0$, the sequences $\{ \qln \}$ and $\{ \dql \}$ converge to the homothetic-parabolic solution $\qb$ and its derivative $\dot{\qb}$ correspondingly, uniformly on $[0,T]$ and on compact subsets of $(0,T]$ correspondingly. 
\end{prop}

Let $\{\Psi_n \in H^1([0,T], \cx^{\bk})\}_{n=1}^{\infty}$ be a sequence of functions, defined as following 
\begin{equation}
 \label{psi} \Psi_n(t) = 
 \begin{cases}
  \qb(t) - \qln(t), & \text{ if } t \in [0, T-\frac{1}{S_n}];\\
  S_n(T-t)(\qb(t) - \qln(t)), & \text{ if } t \in [T- \frac{1}{S_n}, T],
 \end{cases}
\end{equation}
where the sequence $ \{S_n >0\}$ satisfying $\lim_{n \to \infty} S_n = +\infty$. Notice that $\Psi_n(0) =\Psi_n(T)=0$ and by Proposition \ref{uniform converges}, $\Psi_n$ converges uniformly to $0$ on $[0, T]$. Furthermore define a function $\Phi =(\Phi_j)_{j \in \bk}: [0,T] \to \cx^{\bk}$ as following
\begin{equation}
\label{phi} \Phi(t) = \qe(t) - \qb(t), \quad \forall t \in [0,T],
\end{equation}
where $\qe$ is obtained through Lemma \ref{deform2}. Since $\Phi(t)$ is $C^1$ in a neighborhood of $T$ and $\sum_{j \in \bk} m_j \Phi_j(t) = 0$ for any $t \in [0,T]$, the following was proved in \cite[(7.9)]{FT04}.

\begin{prop} \label{psi phi}
	For any $T \in (0, \dl),$ there is a sequence of integers $\{S_n \} \nearrow +\infty $, such that for $\Psi_n$ and $\Phi$ defined as above,  
    $$ \lim_{n \to \pinf} A_L (\xln +\Phi + \Psi_n, T) - A_L (\xln, T)= A_{L_{\bk}}(\qb + \Phi, T) - A_{L_{\bk}}(\qb, T). $$
\end{prop}

The above two results give us a connection between the isolated $\bk$-cluster collision solution $x$ and the homothetic-parabolic solution $\qb$. Now we can prove Proposition \ref{deform1}
\begin{proof} 
 
 [\textbf{Proposition \ref{deform1}}] Choose a $T \in (0, \dl)$, for each $\lmd_n$, let
 $$ \yln (t) = 
 \begin{cases}
  \xln(t) +\Phi(t) + \Psi_n(t), & \text{ if } t \in [0, T]; \\
  \xln(t), & \text{ if } t \in [T, \frac{\dl}{\lmd_n}]. 
 \end{cases}
$$
They are well-defined as $\Phi(T) = \Psi_n(T)=0$. By Proposition \ref{psi phi} and Lemma \ref{deform2},
\begin{equation*} 
\begin{split}
\lim_{n \to \pinf} A_{L}(\yln, \dl/\lmd_n) &- A_{L}(\xln, \dl/\lmd_n)  = \lim_{n \to \pinf} A_{L}(\yln, T) - A_{L}(\xln, T)   \\
 & = A_{\lk}(\qb +\Phi, T) - A_{\lk}(\qb, T) =A_{\bk}(\qe, T)-A_{\bk}(\qb,T) < 0. 
\end{split}
\end{equation*}
Hence for $n$ large enough,
\begin{equation}
 \label{lmdn}  A_{L}(\yln, \dl/\lmd_n) < A_{L}(\xln, \dl/\lmd_n).
\end{equation}
For each $n$, we define a new path: $ y_n(t) = \lmd_n^{\pw} \yln(t/ \lmd_n),$ for any $t \in [0, \dl]$. Then
$$ y_n(t) - x(t) = 
\begin{cases}
\lmd_n^{\pw}[\Phi(t/\lmd_n)+\Psi(t/\lmd_n)], & \text{ if } t \in [0, \lmd_n T]; \\
0, & \text{ if } t \in [\lmd_n T, \dl].
\end{cases}
 $$
This shows, for $n$ large enough, $y_n$ is just a small deformation of $x$ in a small neighborhood of $0$ and 
$$ y_n(0) = x(0) + \lmd^{\pw} \Phi(0) = x(0) + \lmd_n^{\pw} \ep \sg, $$
so $y_n(t)$ is collision-free, for any $t \in (0, \dl]$ and $y(0)$ satisfies statement (c) in Proposition \ref{deform1}. At the same time, a straight forward computation shows
$$ A_L(y_n, \dl) - A_{L}(x, \dl) = \lmd_n^{\frac{2-\al}{2+\al}}[A_L(\yln, \dl /\lmd_n) - A_L(\xln, \dl/\lmd_n)] <0.$$
\end{proof}

Although we only talked about isolated collision happening at $t=0$ so far. By simply reversing time, we can get similar results when an isolated collision occurs at $t =T_0$. At the same, an isolated collision can also occur at $t \in (0,T_0)$. To deal with this, we extend the definition of \textbf{isolated $\bk$-cluster collision solution} given in Definition \ref{cluster collision} to $x(t), t \in [-\dl, \dl]$, by adding an extra condition that the energy being conserved during the isolated collision $x(0)$: 
$$K(\dot{x}(t)) -U(x(t)) \equiv \text{Constant}, \quad  \forall t \in [-\dl, \dl] \setminus \{0\}.$$

Let $s(t)$ be the normalized centered $\bk$-configuration defined as in \eqref{s}. If two sequences $\{t^+_n >0 \}$ and  $\{ t^-_n < 0 \}$ satisfies $\lim_{n \to \infty} t_n^{\pm}=0$ and $\lim_{n \to \pinf}s(t_n^{\pm}) = s^{\pm}$, by Proposition \ref{cc}, $s^{\pm}$ are two central configurations of the $\bk$-body problem. 
Let $\sg =(\sg_k)_{k \in \bk}$ be a normalized centered $\bk$-configuration satisfying \eqref{eqn:sg}. 
\begin{prop} \label{deform3}
 If $\langle \sg_{jk}, s^{\pm}_{jk} \rangle \ge 0$, for any $\{j < k \} \in \bk$, where $\sg_{jk}= \sg_j -\sg_k$ and $s^{\pm}_{jk} = s^{\pm}_j - s^{\pm}_k,$ 
then for $\dl_0, \ep_0>0$ small enough, there is a new path $y \in H^1([-\dl, \dl], \cx)$ satisfying $A_L(y, [-\dl, \dl]) < A_L(x, [-\dl, \dl]),$ and 
        \begin{enumerate}
	 	\item[(a).] $y(t) = x(t), \; \forall t \in[-\dl, -\dl_0]\cup[\dl_0, \dl];$
	 	\item[(b).] $y(t)$ is collision-free, for any $t \in [-\dl, \dl] \setminus \{0\};$
	 	\item[(c).] $y|_{[\dl_0,\dl_0]}$ is a small deformation of $x|_{[\dl_0,\dl_0]}$ with $y(0) = x(0) + \ep_0 \sg$, in particular $y_{k_1}(0) \ne y_{k_2}(0)$, for any $\{k_1 \ne k_2 \} \subset \bk$ and $y_k(0) \ne y_j(0)$, for any $k \in \bk$ and $j \notin \bk$.
	 	\end{enumerate}
\end{prop}
\begin{proof}
	 By applying Proposition \ref{deform1} to $x|_{[-\dl, 0]}$ and $x|_{[0, \dl]}$ correspondingly (for $x|_{[-\dl, 0]}$, one needs to first reverse the time and then shift it by $\dl$), we get two collision-free paths $y^- \in H^1([-\dl, 0], \cx)$ and $y^+ \in H^1([0, \dl], \cx)$ with $y^-(0) = y^+(0) = x(0) + \ep_0 \sg.$ Simply piece them together at $t=0$, we get a path $y$ with the desired properties.  
\end{proof}

\section{Binary collision} \label{bicoll}

The local deformation results obtained in Section \ref{sec local} imposed strong constraints on the possible directions of deformation, which limits their application. In this section, we show such constraints can be substantially relaxed for isolated binary collisions. This result was not available in \cite{Mo99} and is the key property in our proof of the main result. We believe it could be useful in other action minimizing problems as well, for example see \cite{Y15c}, \cite{Y16}.

Following the notations introduced in the previous section, we fix an arbitrary isolated $\bk$-cluster collision solution $x \in H^1([0,\dl], \cx)$ with an isolated $\bk$-cluster collision at $t=0$. However in this section we only consider the planar $N$-body problem ($N \ge 2$ and $d=2$) and the set $\bk$ will only include two different indices. Without loss of generality, we set $\bk = \{2,3\}$ for the rest of this section. 

We need a couple of results regarding the asymptotic behavior of the masses as they approach to the collision. In the rest of the paper, for $t>0$ small enough, by $f(t) \sim Ct^{\beta}$, we mean $f(t) = Ct^{\beta} +o(t^{\beta})$.
\begin{prop}
	\label{sundman} Let $I_k(t)=I_i(x(t))$ be defined as in \eqref{ik}, then for $t>0$ small enough
	$$ \ik(t) \sim (\kp t)^{\frac{4}{2 + \al}}, \quad \dot{I}_{\bk}(t) \sim \frac{4}{2+\al}\kp (\kp t)^{\frac{2-\al}{2 +\al}}, \quad \ddot{I}_{\bk}(t) \sim 4 \frac{2-\al}{(2+\al)^2}{ \kp}^2 (\kp t)^{\frac{-2\al}{2+\al}}. $$
	The constant $\kp$ is the same as the one given in \eqref{qb}.
\end{prop}
This is the well-known Sundman's estimates (a proof can be found in \cite[(6.25)]{FT04}). 

Next we need to know the asymptotic directions of the masses as they approach to the collision. For this, it's better to use polar coordinates: 
\begin{equation}
\label{polar} q(t) = (q_2(t), q_3(t)):= (\rho_2(t) e^{\tht_2(t)}, \rho_3(t) e^{\tht_3(t)}).
\end{equation}
Here $q(t)$ is defined as in \eqref{q}. Then $m_2 q_2(t) + m_3 q_3(t) =0$ implies
\begin{equation}
\label{rho} \tht_2(t) = \tht_3(t) +\pi, \;\; \rho_3(t) = \frac{m_2}{m_3}\rho_2(t) = \sqrt{\frac{m_2}{m_3(m_2 + m_3)}} \ik^{\frac{1}{2}}(t).
\end{equation}
The following result can be seen as a generalization of Proposition \ref{cluster collision} in the case of binary collision. 
\begin{prop}
	\label{2cc}
	There are $\tht_3^+$ and $\tp_2 = \tp_3 + \pi$ with the following limits hold
	\begin{enumerate}
	 \item $ \lim_{t \to 0^+} \tht_2(t) = \tht^+_2, \; \lim_{t \to 0^+}\tht_3(t) = \tht^+_3;$
	 \item $ \lim_{t \to 0^+} \dot{\tht}_2(t) = \lim_{t \to 0^+} \dot{\tht}_3(t) = 0.$
	 \end{enumerate} 
	\end{prop}

\begin{proof}
	Consider the relative angular momentum between $m_3$ and $m_2$
	$$\jk(t) = [x_3(t) - x_2(t)] \wedge [\dx_3(t) - \dx_2(t)]. $$
	By the triangle inequality
    \begin{equation}
    \label{eqn:jk} |\jk|^2  \le |x_3 - x_2|^2 |\dx_3 - \dx_2|^2 \le |x_2 - x_3|^2 (|\dx_3|^2 + |\dx_2|^2).
    \end{equation}
    Because the energy of the $\bk$-body sub-system is bounded (for a proof see \cite[(6.24)]{FT04})
    $$ \ey[m_2 |\xd_2(t)|^2 +m_3 |\xd_3(t)|^2] - \uk(x(t)) \le C, \;\; \forall t \in (0, \dl),$$
    the fact $m_j>0$, for any $j$, implies 
    $$ |\xd_2(t)|^2 + |\xd_3(t)|^2 \le C_1 + C_2 \uk(x(t)), \;\; \forall t \in (0, \dl).$$
    As $2-\al>0$, plugging this into \eqref{eqn:jk} shows
	\begin{equation}
	\label{jkl} \lim_{t \to 0^+} \jk(t) = 0.
	\end{equation}

	At the same time
	$$\dot{\jk} = (x_3 - x_2) \wedge (\ddot{x}_3 - \ddot{x}_2) = x_3 \wedge \ddot{x}_3 + x_2 \wedge \ddot{x}_2 - x_3 \wedge \ddot{x}_2 - x_2 \wedge \ddot{x}_3. $$
    By \eqref{N body}, a simple computation shows 
	\begin{align*}
	 \dot{\jk}  & = (x_3 \wedge x_2) \sum_{ k \in \bn \setminus \bk} m_k( \frac{1}{|x_2 - x_k|^{\al+2}} - \frac{1}{|x_3 - x_k|^{\al + 2}})  \\
	 & + (x_3 -x_2) \wedge \sum_{k \in \bn \setminus \bk} m_k x_k ( \frac{1}{|x_2 - x_k|^{\al+2}} - \frac{1}{|x_3 - x_k|^{\al + 2}}). 
	\end{align*}
	Combing this with 
	$$ x_3 \wedge x_2 = \frac{m_2 - m_3}{m_2} q_3 \wedge x_0, \quad x_3 - x_2 = \frac{m_2 + m_3}{m_2 } q_3,$$
    we find, for any $t \in (0,\dl)$
	$$ | \dot{\jk}(t)| \le C_3 |x_3(t) \wedge x_2(t)| + C_4 |x_3(t) -x_2(t)| \le C_5 |q_3(t)| =C_5 \rho_3(t). $$
	By Proposition \ref{sundman} and \eqref{rho}, one gets $\rho_3(t) \sim Ct^{\frac{2}{2+\al}}$. Therefore
	$$ |\dot{\jk}(t)| \le C_6 t^{\frac{2}{2 + \al}}, \quad \forall t \in (0, \dl).$$
	Along with \eqref{jkl}, it shows
	$$ |\jk(t)| \le \int_0^t |\dot{\jk}(\tau)| \,d\tau \le C_5 t^{\frac{4+\al}{2+\al}}. $$

	Meanwhile in polar coordinates, 
	\begin{equation}
	\label{jk} \jk(t) =[q_3(t) -q_2(t)]\wedge[\qd_3(t) -\qd_2(t)] = \frac{(m_2 + m_3)^2}{m_2^2} \rho_3^2(t) \td_3(t).
	\end{equation}
	This implies
	$$ | \td_3(t)| \le C_6 t^{\frac{\al}{2 + \al}}. $$
	As a result $ \lim_{t \to 0^+} \td_3(t) = 0,$ and there must be a $\tp_3$ with $ \lim_{t \to 0^+} \tht_3(t) = \tp_3.$ The rest of the proposition follows from $\tht_2(t) = \tht_3(t) +\pi$ given in \eqref{rho}.	
\end{proof}
Let $s^+=(s^+_2, s^+_3) = (\rb_2 e^{\tp_2}, \rb_3 e^{\tp_3})$, where
\begin{equation}
\label{rhop}  \rb_2 = \sqrt{\frac{m_3}{m_2 (m_2 + m_3)}},  \quad \rb_3 = \sqrt{\frac{m_2}{m_3 (m_2 + m_3)}}. 
\end{equation}
Then $s^+$ is a normalized central configuration of the $\bk$-body problem. We point out that in the 2-body problem, every normalized centered configuration is a central configuration and it is unique up to rotation. By Proposition \ref{2cc}, $\lim_{t \to 0} s(t) = s^+$, where $s(t) = \ik^{-\ey} (t) q(t)$.

Given a normalized centered $\bk$-configuration $\sg =(\sg_2, \sg_3)$. Put it in polar coordinates: $\sg = (\sg_2, \sg_3) =(\rb_2 e^{i \phi_2}, \rb_3 e^{i \phi_3})$ with $ \phi_2 = \phi_3 + \pi.$

\begin{prop}
 \label{deform21}
 If the following condition holds
 \begin{equation} \label{eqn:binary}
 \begin{cases}
 |\phi_3 -\tp_3| \le \pi, &\text{ when } \al \in (1,2); \\
 |\phi_3 -\tp_3| < \pi, &\text{ when } \al =1,
 \end{cases}
 \end{equation}
 then for $\dl_0, \ep_0>0$ small enough, there is a new path $y \in H^1([0,\dl], \cx)$ satisfying $A_L(y, \dl) < A_L(x, \dl)$ and 
 \begin{enumerate}
 	\item[(a).] $y(t) = x(t), \; \forall t \in [\dl_0, \dl];$
 	\item[(b).] $y(t)$ is collision-free, for any $t \in (0, \dl];$
 	\item[(b).] $y|_{[0,\dl_0]}$ is a small deformation of $x|_{[0,\dl_0]}$ with $y(0) = x(0) + \ep_0 \sg$, in particular $y_2(0) \ne y_3(0)$ and $y_{k}(0) \ne y_j(0)$, for any $k \in \bk$ and $j \notin \bk$.
 	\end{enumerate}
 \end{prop}
  \begin{rem}
 \label{rem: Gordon} For Newtonian potential, $\al=1$, by Gordon's classical result on Kepler problem \cite{Go77}, the result we obtained in Proposition \ref{deform21} is optimal, i.e., the corresponding result does not hold for $|\phi_3-\tht_3^+|=\pi$, when $\al =1$. 
 \end{rem}

 We notice that condition \eqref{eqn:dfm dir} in Proposition \ref{deform1} holds, when $\phi_3 \in [\tp_3- \frac{\pi}{2}, \tp_3 + \frac{\pi}{2}]$ and fails, when $\phi_3 \in  [\tp_3- \pi, \tp_3 + \pi] \setminus [\tp_3- \frac{\pi}{2}, \tp_3 + \frac{\pi}{2}]$. Therefore Proposition \ref{deform21} is a substantial improvement of Proposition \ref{deform1} in the case of binary collision. 

The proof of the above result follows the same idea used in the proof of Proposition \ref{deform1}. Let $\qb : [0, \pinf) \to \cx^{\bk}$ be the homothetic-parabolic solution defined in \eqref{qb} with $\sbb$ replaced by $s^+ =(s^+_2, s^+_3) = (\rb_2 e^{\tp_2}, \rb_3 e^{\tp_3})$. The next result generalizes Lemma \ref{deform2} in the case of binary collision.

\begin{lm}
	\label{deform22} For any $T>0$, if condition \eqref{eqn:binary} in Proposition \ref{deform21} holds, then there is a collision-free path $z \in H^1([0,T], \cx^{\bk})$ satisfying $A_{\lk}(z, T) < A_{\lk}(\qb, T)$ and
	\begin{enumerate}
		\item[(a).] $z(T) = \qb(T)$ and  $\text{Arg}(z_k(T)) = \tp_k,$ for $k =2,3;$
		\item[(b).] $z(0) = \ep \sg = \ep(\rb_2 e^{i \phi_2}, \rb_3 e^{i \phi_3}),$ for some $\ep > 0.$
	\end{enumerate}
\end{lm}
A proof of Lemma \ref{deform22} will be given in the Appendix. Now we will use it to prove Proposition \ref{deform21}. 
\begin{proof}
	
	[\textbf{Proposition \ref{deform21}}] By Lemma \ref{deform22} there is a $ z \in H^1([0, T], \cx^{\bk})$ satisfies all the properties listed there. Because $z(T) = \qb(T)$, we can extend the domain of $z$ to $[0,T]$ by simply attaching $\qb|_{[T, T]}$ to it. We will still call the new curve $z$. 
	
	Define $\Phi \in H^1([0, T], \cx^{\bk})$ by $\Phi(t) = z(t) - \qb(t)$, just as in \eqref{phi} with $\qe$ replaced by $z$, and $\Psi_n \in H^1([0, T], \cx^{\bk})$ as in \eqref{psi}. The rest of the proof is the same as Proposition \ref{deform1} and we will not repeat it here.
\end{proof}

Like Section \ref{sec local}, although we only talked about an isolated binary collision happening at $t=0$. By reversing the time, similar results as above hold when an isolated binary occurs at $t=T_0.$ Similarly a stronger result than Proposition \ref{deform3} can also be obtained in the case of an isolated binary collision. However this will not be needed in this paper, as Proposition \ref{deform3} will be enough for us. 

\section{Proof of Proposition \ref{prop}} \label{weak}
The existence of such an action minimizer $x \in \Om$ following from the coercivity and lower semi-continuous of $A_L$, for the details see \cite{Mo99}. What's left is to prove $x(t)$ is collision-free, for any $t \in [0,T_0]$. While Marchal's approach can be used to show this when $t \in (0, T_0)$ (although it does not work at $t=0$ or $T_0$), we show the alternative approach proposed by Montgomery in \cite{Mo99} will also work. 

Because the Lagrangian $L_{\mcc}$ in the shape space is symmetric with respect to the syzygy plane $\{ w_3 = 0\}$ and the isosceles plane $c(M_3)$ (as $m_1=m_2$), we may assume  $\pi(x(t))$ (the projection of $x(t)$ to the shape space)  never crosses the syzygy plane or $c(M_3)$, i.e.,
\begin{equation}
\label{iso} w_3(x(t)) \ge 0, \;\; |x_3(t) - x_1(t)| \ge |x_3(t) - x_2(t)|, \;\; \forall t \in [0,T_0].
\end{equation}

As a minimizer, $x$ can only have isolated collisions and its energy is conserved through collisions (for a proof see \cite{Mo99} or \cite{Ce02}). By a  contradiction argument, let's assume $x(t_0)$ is an isolated collision for some $t_0 \in [0, T_0]$.
For the three body problem, a collision is either a triple collision or a binary collision, we will discuss them separately. 

First let's assume $x(t_0)$ is a triple collision. Let $s(t) = I(t)^{-\ey} x(t)$ be a normalized configuration. By a result of Siegel (\cite{Si41}), when $t$ approaches to $t_0^+$ (or $t_0^-$), $s(t)$ converges to a normalized central configuration $s^+$ (or $s^-$). 
\begin{lm}
	\label{lm1} If $t_0 \in (0,T_0),$ $x(t_0)$ can't be a triple collision. 
\end{lm}

\begin{proof}
	Assume $x(t_0) =0$ is a triple collision. By \eqref{iso}, $\pi(x(t))$ never goes below the syzygy plane. Hence $s^{\pm}$ can not be $L^-$. This left us with the following three cases.

	\textbf{Case 1}: both $s^+$ and $s^-$ are the Lagrange configuration $L^+$. Then after rotating $x|_{[t_0,T_0]}$ by be a proper angle in the inertial plane with respect to the origin, we can have $s^+ = s^-.$ 
	
	\textbf{Case 2}: $s^-$ is the Lagrange configuration $L^+$ and $s^+$ is an Euler configuration (the alternative case is similar). There are three different Euler configurations and by rotating $x|_{[t_0,T_0]}$ with a proper angle, we may put $s^-$ and $s^+$ in relative positions given by the pictures in Figure \ref{8-2}, where $j^{\pm}$ indicates the position of $m_j$ in $s^{\pm}$. The straight line containing the Euler configuration is perpendicular to one of the sides of the equilateral triangle as indicated in the pictures. 
	
	\begin{figure}
		\centering
		\includegraphics[scale=0.7]{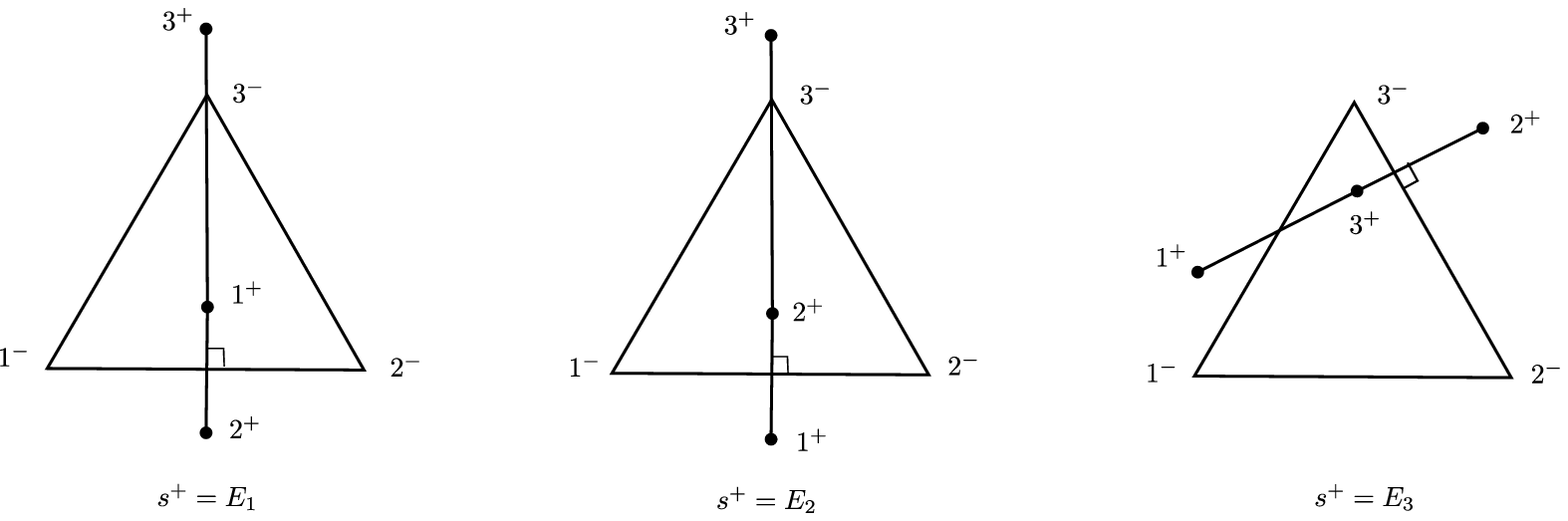}
		\caption{}
		\label{8-2}
		
	\end{figure}
	 
	\textbf{Case 3}: both $s^-$ and $s^+$ are Euler configurations. If $s^-$ and $s^+$ are the same type, again we may assume $s^+ = s^-$; otherwise like in \textbf{Case 2}, we may assume the relative positions of $s^+$ and $s^-$ are given by the pictures in Figure \ref{8-3}, where the straight lines containing the Euler configurations are perpendicular to each other.  
	
    \begin{figure}
    	\centering
    	\includegraphics[scale=0.7]{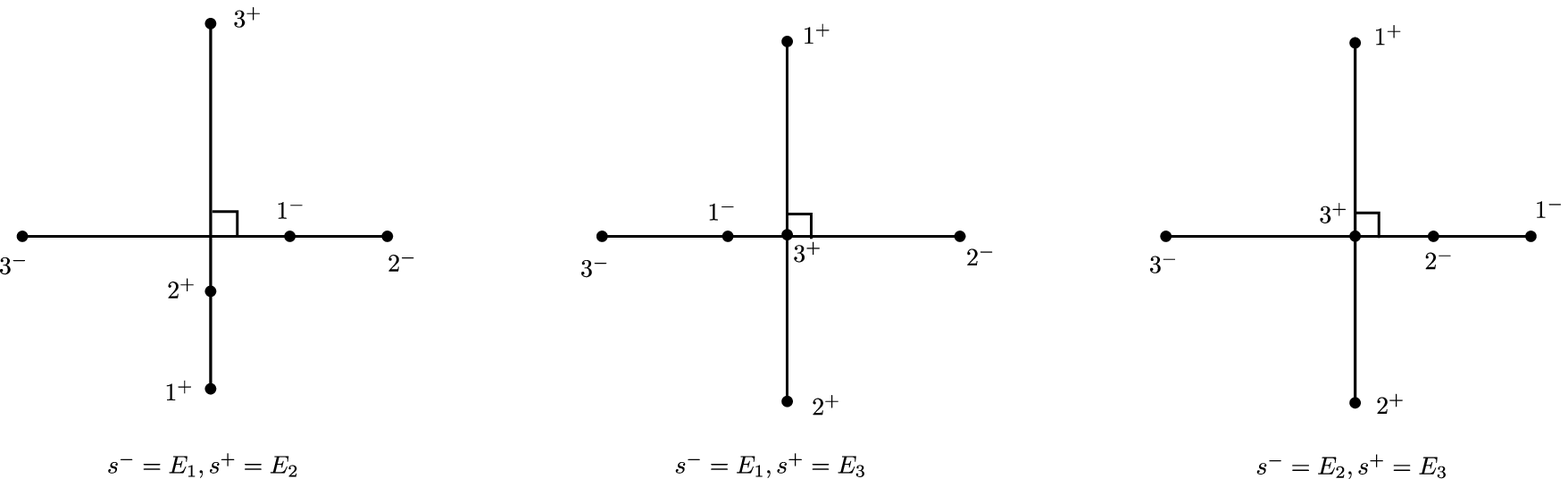}
    	\caption{}
    	\label{8-3}
    \end{figure}
    
	For all the relative positions of $s^{\pm}$ shown in Figure \ref{8-2} and \ref{8-3}, we can see $\langle s^+_{jk}, s^-_{jk} \rangle \ge 0$, for any $1 \le j < k \le 3$. This is automatically true when $s^-=s^+.$    
	
	This allows us to use Proposition \ref{deform3} (for example let $\sg =s^+$) to make a local deformation of $x$ near $t=t_0$ and get a new path $x^\ep \in \Om$ with $A_L(\xe, T_0) < A_L(x, T_0)$, which is absurd.  
		
\end{proof}

\begin{lm}
	\label{lm2} If $t_0 \in \{ 0, T_0\}$, $x(t_0)$ can't be a triple collision. 
\end{lm}

\begin{proof}
	Assume $x(T_0)=0$ is a triple collision. Let $\sg$ be a normalized central configuration $E_3$ with its relative positions with respect to $s^-$ as indicated in the third picture of Figure \ref{8-2}. Then $\langle \sg_{jk}, s^-_{jk} \rangle \ge 0$, for any $1 \le j < k \le 3$. By Proposition \ref{deform1}, there is path $x^{\ep} \in \Om$ (a local deformation of $x$ near $t_0$) satisfying $A_L(x^{\ep},T_0)<A_L(x,T_0)$, which is absurd. Notice that for the $\sg$ we chosen, $\pi(x^{\ep}(T_0)) \in c(E_3)$. 

	For $t_0 =0$, everything is the same except to choose $\sg$ as a normalized central configuration $E_2$ this time.
\end{proof}

Now let's assume $x(t_0)$ is a binary collision or an isolated $\bk$-cluster collision with $\bk =\{1,2\}, \{2,3\}$ or $\{1,3\}$. Correspondingly we define $s(t)$ as a normalized centered $\bk$-configuration given by \eqref{s}. Recall that if $\bk= \{2,3\}$,
$$ x_0(t) = [m_2 x_2(t) + m_3 x_3(t)]/m_0, \;\; m_0 = m_2 +m_3;$$
$$ q(t) = (q_2(t), q_3(t)) = (x_2(t)-x_0(t), x_3(t)-x_0(t));$$
$$\ik(t) = m_2|q_2(t)|^2 + m_3 |q_3(t)|^2.$$

By Proposition \ref{2cc}, $s(t) = \ik^{-\ey}(t)q(t)$ converges to $s^{\pm}$, a normalized central configuration of the $\bk$-body problem, as $t$ converges to $t^{\pm}_0$. 

\begin{lm}
	\label{lm3} If $t_0 \in (0,T_0)$, then $x(t_0)$ can't be a binary collision. 
\end{lm}

\begin{proof}
	We will give details for the case: $\bk = \{2,3\}$, while the others are similar.
	Choose a $\dl>0$ small enough such that $x(t_0)$ is the only collision in $[t_0 -\dl, t_0+\dl]$. In polar coordinates, set
	$$ q(t) = (q_2(t), q_3(t)) = (\rho_2(t) e^{i\tht_2(t)}, \rho_3(t) e^{i \tht_3(t)}).$$
	By Proposition \ref{2cc}, there are $\tht^{\pm}_3 \in [0, 2\pi)$ and $\tht_2^{\pm} = \tht^{\pm}_3 +\pi$, such that 
	$$\lim_{t \to t_0^{\pm}} \tht_j(t) = \tht_j^{\pm}, \;\; \forall j \in \{2,3\}.$$
	As a result $s^{\pm}= (s^{\pm}_2, s^{\pm}_3)= (\rb_2 \tht_2^{\pm}, \rb_3 \tht_3^{\pm})$, where $\rb_2, \rb_3$ are defined in \eqref{rhop}. 

	Since $x(t_0)$ is a collinear configuration, let's assume all the three masses lying on the real axis at the moment $t_0$ with $m_1$ on the negative direction. We claim 
	$$\tht^{\pm}_3 \in [0, \pi], \;\; \tht_2^{\pm} = \tht_3^{\pm}+\pi \in [\pi, 2\pi]. $$ 
	Because if $\tht^+_3 \in (\pi, 2\pi)$, then for $t-t_0>0$ small enough, $x(t)$ is not collinear and the triangle formulated by the masses will have the same orientation as $L^-$. This means $\pi(x(t))$ is below the syzygy plane, which violates \eqref{iso}. The argument for $\tht^-_3$ is the same.
	          
	Let $\sg =(\sg_2, \sg_3) = (\rb_2 e^{\phi_2}, \rb_3 e^{\phi_3})$ be a normalized centered $\bk$-configuration with 
        \begin{equation*}
         \phi_2 = \ey (\tht^-_2 + \tht^+_2); \;\; \phi_3 = \ey ( \tht^-_3 + \tht^+_3).
        \end{equation*}
        Because $|\tht^+_j-\tht^-_j| \le \pi$ for any $j \in \{2,3\}$, we have $\langle s^{\pm}_{23}, \sg_{23} \rangle \ge 0 $. Then by Proposition \ref{deform3}, we can get a new path in $\Om$ with action value strictly smaller than $x$'s and get a contradiction (see picture $(a)$ in Figure \ref{8-5}). 
	\begin{figure}
		\centering
		\includegraphics[scale=0.8]{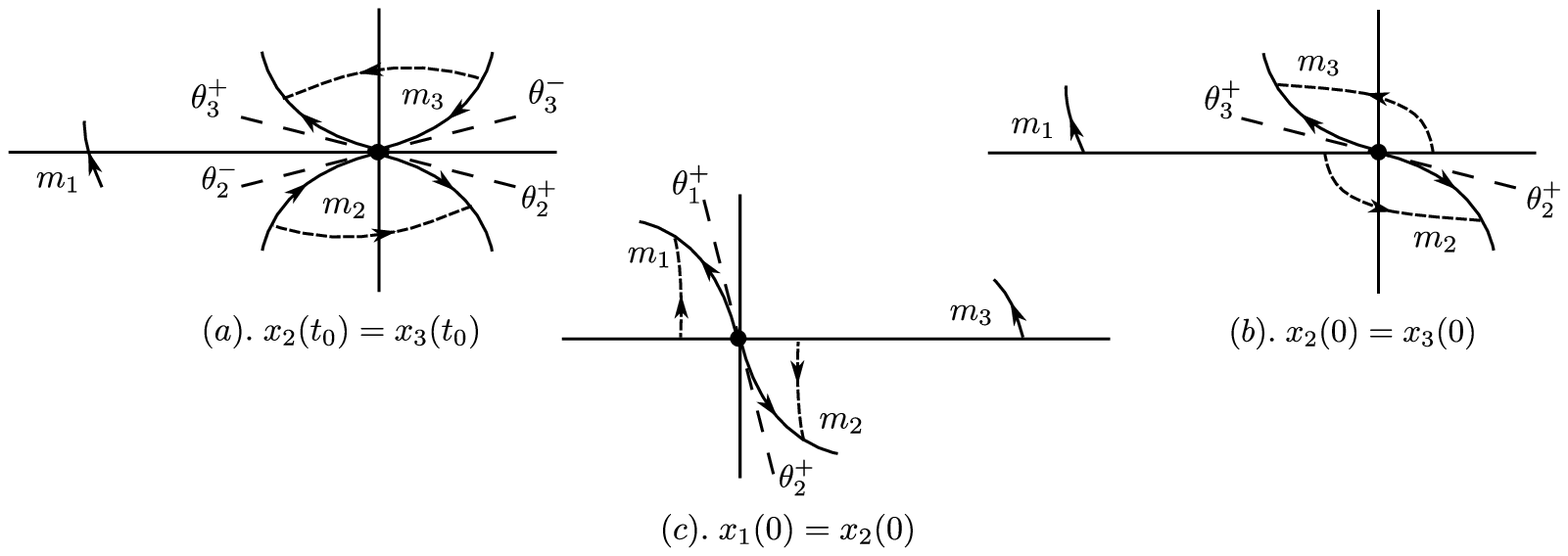}
		\caption{}
		\label{8-5}
	\end{figure}         
\end{proof}

\begin{lm} \label{lm4}
 If $t_0 \in \{ 0, T_0\}$, $x(t_0)$ can't be a binary collision. 
\end{lm}

\begin{proof}
 When $t_0 =T_0$, by the definition of $\Om$, the only possible collision is a triple collision. The result holds automatically.
 
 Assume $x(0)$ is a binary collision, by the definition of $\Om$, it is either $b_{12}$ or $b_{23}$.  
 
 First, let's say $x(0)$ is a $b_{23}$ binary collision, so it is an isolated $\bk$-cluster collision with $\bk = \{2,3 \}.$ Following the same notations and argument used in the proof of the previous lemma, we may assume all the three masses are on the real axis at $t=0$ with $x_1(0) < 0 < x_2(0) = x_3(0)$ and   
 \begin{equation} \label{in1}
 \tht^+_3 \in [0, \pi], \; \tht^+_2 = \tht_3^+ +\pi \in [\pi, 2\pi].
 \end{equation}
 Let $\sg = (\sg_2, \sg_3) = (\rb_2 e^{i \phi_2}, \rb_3 e^{i \phi_3})$ with $ \phi_3 = 0, \phi_2 = \pi.$ Then $ 0 \le \tht^+_3 - \phi_3 \le \pi$. 

 As a result, we can use Proposition \ref{deform21} to get a new path $y \in \Om$ with $A_L(y,T_0)<A_L(x,T_0)$ (see picture $(b)$ in Figure \ref{8-5}) and this is a contradiction. This is shows $x(0)$ can't be a $b_{23}$ binary collision.  
 
 Now let's assume $x(0)$ is a $b_{12}$ binary collision, so $x(0)$ is an isolated $\bk$-cluster collision with $\bk = \{1,2\}.$ Like before we assume all the three masses are in the real axis at $t=0$ with $x_1(0)=x_2(0) <0 < x_3(0)$. 

 Now we let $q_j(t)$, $j=1,2$, represents relative position of $m_j$ with respect to the center of mass of $m_1$ and $m_2$ and in polar coordinates: $q_j(t) =\rho_j(t) e^{i \tht_j(t)}$. By Proposition \ref{2cc}, there exist $\tht_1^+ \in [0, 2\pi)$ and $\tht^+_2=\tht_1^+ +\pi$ satisfying 
 $$\lim_{t \to 0^+} \tht_j(t) = \tht_j^+, \;\; \forall j \in \{1,2 \}. $$
 By our assumption, $\pi(x(t))$ is never below the syzygy plane. Similar argument as before shows $\tht_1^+ \in [0, \pi]$. Meanwhile the assumption that $\pi(x(t))$ never crosses $c(M_3)$ (the second inequality in \eqref{iso}) further implies $\tht_1^+ \in [\pi/2, \pi]$. Then $\tht_2^+ \in [3\pi/2, 2\pi]$. 

 Let $\tilde{\sg} = (\tilde{\sg}_1, \tilde{\sg}_2) = (\tilde{\rho}_1 e^{i \tilde{\phi}_1}, \tilde{\rho}_2 e^{1 \tilde{\phi}_2})$ be a normalized centered $\bk$-configuration with $\tilde{\phi}_1 = \pi$ and $\tilde{\phi}_2 =2\pi$. Then by Proposition \ref{deform1}, we can get a new path $x^{\ep} \in \Om$ with $A_L(x^{\ep}, T_0) < A_L(x, T_0)$ (see picture $(c)$ in Figure \ref{8-5}), which is absurd. This shows $x(0)$ can not be a $b_{12}$ collision either.   
 \end{proof}

Up to now we have proved as a minimizer $x$ must be collision-free and this finishes our proof of Proposition \ref{prop}.


\section{Proof of Proposition \ref{prop 2}} \label{sec newton}
The existence of an action minimizer in $\Om$ can be proved just like before. Let $x \in \Om$ be such a minimizer. Like in Section \ref{weak}, we assume it satisfies \eqref{iso}. 

Recall that Proposition \ref{deform21} was used only once in Section \ref{weak} to prove $x(0)$ is free of $b_{23}$ binary collision. In the other cases we used Proposition \ref{deform1} and \ref{deform3}, whose results hold for any $\al \in [1,2)$, so those results will still hold for $\al =1$. Therefore the only possible collision is a $b_{23}$ binary collision at $t =0$. For the rest we assume $\bk=\{2,3\}.$ 

\begin{lm} \label{lm: b23}
If $x(0)$ has a $b_{23}$ binary collision, then $x(t) \in \rr^3, \forall t \in [0,T_0]$ and its a quarter of a Schubart solution. 
\end{lm}
\begin{proof}
We can still use the same notations and assumptions set up in the first half of the proof of Lemma \ref{lm4}. By condition \eqref{eqn:binary} in Proposition \ref{deform21}, $x(0)$ can have a $b_{23}$ binary collision, only when $\tht_3^+ = \pi$. As otherwise, Proposition \ref{deform21} can still be used like in the proof of Lemma \ref{lm4} to reach a contradiction, even when $\al=1$. Let $x_0(t)$ is the center of mass of $m_2$ and $m_3$. Set  
$$ x_j(t) = u_j(t) + i v_j(t), \;\; \text{ for } j \in \{0, 1,2,3 \} \;\; \text{ with } u_j(t), v_j(t) \in \rr. $$
We claim 
\begin{equation}
 \label{eqn:v} \dot{v}_1(0) = \lim_{t \to 0^+} \dot{v}_2 (t) = \lim_{t \to 0^+} \dot{v}_3(t) = 0.
 \end{equation} 
Let $m_0 = m_2 +m_3$, then
\begin{equation}
\label{eqn:x0} m_0x_0(t) +m_1x(t)= 0, \;\; \forall t \in (0,T_0).
\end{equation}
As an action minimizer, the angular momentum of $x$ vanishes:
$$ J(x(t)) = \sum_{j=1}^3 m_j x_j(t) \wedge \xd_j(t)=0, \;\; \forall t \in (0, T_0).$$ Rewrite the angular moment as following 
\begin{equation*}
J(x) = \sum_{j=0}^{1} m_j x_j \wedge \xd_j + \sum_{k=2}^{3}m_k(x_k -x_0)\wedge (\xd_k -\xd_0) =0.
 \end{equation*} 
In the proof of Proposition \ref{2cc}, we showed $\lim_{t \to 0^+} \jk(t) =0,$ where
\begin{equation*}
 \jk(t) = (x_3(t) -x_2(t)) \wedge (\xd_3(t) -\xd_2(t)).
\end{equation*}
A simple computation shows 
\begin{equation*}
\frac{m_2 m_3}{m_0}J_{\bk}(t) = \sum_{k=2}^{3}m_k(x_k(t) -x_0(t))\wedge (\xd_k(t) -\xd_0(t)).
\end{equation*} 
Combining the above results, we get
\begin{equation*}
\label{l3} \lim_{t \to 0^+}  m_0 x_0(t) \wedge \dx_0(t) + m_1 x_1(t) \wedge \dx_1(t) =0.
\end{equation*} 
Together with \eqref{eqn:x0}, they imply 
$$ \lim_{t \to 0^+} x_1(t) \wedge \dx_1(t) = \lim_{t \to 0^+}[u_1(t) \dot{v}_1(t) - v_1(t) \dot{u}_1(t)] =0. $$
Because $m_1$ is on the negative axis at $t=0$ without involving in the collision,
$$ u_1(0) < 0, \; v_1(0)=0, \; |\dot{u}_1(0)| < +\infty. $$
This means the following must hold
$$ \dot{v}_1(0) = \lim_{t \to 0^+} \dot{v}_1(t) =0.$$
\eqref{eqn:x0} also implies $m_0 \dot{v}_0(t) +m_1 \dot{v}_1(t) =0$. Therefore 
$$ \dot{v}_0(0) = \lim_{t \to 0^+} \dot{v}_0(t) = \lim_{t \to 0^+} m_1\dot{v}_1(t)/m_0= 0. $$

To estimate $\dot{v}_2(t)$ and $\dot{v}_3(t)$, it is better to use polar coordinates
$$ x_j - x_0 = (u_j -u_0) + i (v_j -v_0)= \rho_j e^{i\tht_j}, \;\; \forall j \in \{2, 3\}. $$
Then 
$$ \dot{v}_3 = \dot{v}_0 + \dot{\rho}_3 \sin \tht_3 + \rho_3 \dot{\tht}_3 \cos \tht_3. $$
By Proposition \ref{sundman}
$$ \rho_3(t) \sim C_1 t^{\frac{2}{3}}, \; \dot{\rho}_3(t) \sim C_2 t^{-\frac{1}{3}}.$$
Meanwhile Proposition \ref{2cc} implies $\lim_{t \to 0^+} \dot{\tht}_3(t) = 0$. Therefore 
$$ \lim_{t \to 0^+} \rho_3(t) \dot{\tht}_3(t) \cos (\tht_3(t)) = 0 $$
Recall that $\tht^+_3=\lim_{t \to 0^+} \tht_3(t) = \pi$, by Taylor expansion, 
$$ | \dot{\rho}_3(t) \sin(\tht_3(t))| \le C_3 t^{\frac{2}{3}}, \text{ for } t>0 \text{ small enough}. $$
As a result, $\lim_{t \to 0^+} \dot{v}_3(t) = 0.$ Similarly we can show $\lim_{t \to 0^+} \dot{v}_2 (t) = 0.$ This establishes \eqref{eqn:v}. Since $x(0) \in \rr^3$, this means $x(t) \in \rr^3$ for any $t \in [0,T_0]$. The rest follows from Proposition \ref{Venturelli} and the fact that $x$ is an action minimizer in $\Om$.
\end{proof}

Proposition \ref{prop 2} follows directly from Lemma \ref{lm: b23}. 

\section{Appendix} \label{sec:app}
A proof of Lemma \ref{deform22} will be given here. Recall that $\bk=\{2,3\}$, so after blow up, we are dealing with the planar $2$-body problem. It is equivalent to the \emph{Kepler-type problem} or the \emph{$1$-center problem}, which describes the motion of a massless particle in a plane under the attraction of mass $M$ fixed at the origin. The position function of the massless particle satisfies the following equation
\begin{equation} \label{kepler problem}
     \ddot{\gm}(t)= \nabla V(\gm(t)), \;\;  V(\gm(t)) := \frac{M}{\al |\gm(t)|^{\al}}.
\end{equation}
It is the Euler-Lagrange equation of the following action functional
$$A_{\lb}(\gm, [T_1, T_2]) = \int_{T_1}^{T_2} \lb(\gm, \dot{\gm})\,dt; \;\; \lb(\gm, \dot{\gm})= \ey |\dot{\gm}|^2 +V(\gm).$$
For simplicity, let $A_{\lb}(\gm, T)= A_{\lb}(\gm, [0,T]),$ for any $T>0$. 

Given an arbitrary pair of angles $\psi^{\pm}$, there is a corresponding \emph{parabolic collision-ejection solution} $\gb:\rr \to \cc$ defined as following:
\begin{equation*}
\gb(t) = 
\begin{cases}
\big( \frac{2+\al}{\sqrt{2\al}} \sqrt{M}|t|\big)^{\pw} e^{i \psi^-} & \text{ if } t \le 0, \\
\big(\frac{2+\al}{\sqrt{2\al}} \sqrt{M}|t| \big)^{\pw} e^{i \psi^+} & \text{ if } t > 0.
\end{cases}
\end{equation*}
A straight forward computation shows, $\gb(t)$ satisfies \eqref{kepler problem} with zero energy, for any $t \ne 0$.

\begin{prop}
	\label{coll1} If the following condition holds
	$$ \begin{cases}
	|\psi^+ - \psi^-| \le 2 \pi, & \text{ when } \al \in (1, 2); \\
	|\psi^+ - \psi^-| < 2 \pi, & \text{ when } \al =1, 	
	\end{cases} $$
	then there is a solution of equation \eqref{kepler problem}, $\gm \in C^2([-T, T], \cc \setminus \{0 \})$, satisfying 
	\begin{enumerate}
	\item[(a).] $\gm(\pm T) = \gb(\pm T), \;\; \text{Arg}(\gm(\pm T)) = \psi^{\pm};$
	\item[(b).] $A_{\lb}(\gm, [-T, T]) < A_{\lb}(\gb, [-T, T]).$
	\end{enumerate}	
\end{prop}

A proof of this result can be found in \cite{Y15a} following the ideas given in \cite{TV07} and \cite{ST12}. When $\al=1$, an alternative proof can be found in the appendix of \cite{FGN11}, where it was attributed to Marchal.  

By the symmetries of the Kepler-type problem, $\gm$ is symmetric with respect to the line passing through the origin and $\gm(0)$ with $\text{Arg}(\gm(0)) = \ey(\psi^- + \psi^+)$. Therefore $ A_{\lb}(\gm, T) < A_{\lb}(\bar{\gm}, T)$. This implies
\begin{cor} \label{coll2}
	Given an angle $\psi$, if the following condition holds
	$$ \begin{cases}
	 |\psi - \psi^+| \le \pi, & \text{ when } \al \in (1,2);\\
	 |\psi - \psi^+| < \pi, & \text{ when } \al =1, 
	\end{cases}$$
	then there is a solution of equation \eqref{kepler problem}, $\gm \in C^2([0, T], \cc \setminus \{0\})$, satisfying 
 	\begin{enumerate}
 	\item[(a).] $\gm(T) = \gb(T), \; \text{Arg}(\gm(0)) = \psi, \; \text{Arg}(\gm(T)) = \psi^+;$
 	\item[(b).] $A_{\lb}(\gm, T)< A_{\lb}(\gb, T)$.
 	\end{enumerate}
\end{cor}
Use the above result, now we can give a proof of Lemma \ref{deform22}.
\begin{proof}	

	[\textbf{Lemma \ref{deform22}}] 
	Recall that the homothetic-parabolic solution $\qb|_{[0, \infty)}$ is defined as following:
	$$ \qb(t)=(\qb_2(t), \qb_3(t)) = (\kp t)^{\frac{2}{2+\al}} s^+ =(\kp t)^{\frac{2}{2+\al}} (\rb_2 e^{i\tht_2^+}, \rb_3 e^{i\tht_3^+}),$$
	where $s^+=(\rb_2 e^{i\tht_2^+}, \rb_3 e^{i\tht_3^+})$ is a normalized central configuration of the $\bk$-body problem ($\bk =\{2,3\}$) with $\tht_2^+ = \tht_3^+ +\pi$ and $\rb_j$, $j =2,3$, defined in \eqref{rhop}. Then	
	\begin{equation}  \label{eqnl}
	\begin{split}
	A_{\lk}(\qb, T) & = \int_0^{T} \ey m_2 |\dot{\qb}_2(t)|^2 + \ey m_3 |\dot{\qb}_3(t)|^2 + \frac{m_2 m_3 }{\al |\qb_2(t) - \qb_3(t)|^{\al}} \, dt \\
	& = \frac{m_0m_3}{m_2} \int_0^T \ey |\dot{\qb}_2(t)|^2 + m_2 (\frac{m_2}{m_0})^{1+\al} \frac{1}{\al |\qb_3(t)|^{\al}} \,dt, \\	
	\end{split}
	\end{equation}		
	where $m_0 = m_2 +m_3$. Let $M=m_2(m_2/m_0)^{1 +\al}$ be the mass of the Kepler-type problem, then
	\begin{equation}
	\label{ineq3}  A_{\lk}(\qb, T) = \frac{m_0 m_3}{m_2}A_{\lb}(\qb_3, T).
	\end{equation}
    Since the energy of $\qb(t)$ is zero, we find 
    \begin{equation}
    \label{kp} \kp = \frac{2+\al}{\sqrt{2\al}}m_0^{-\frac{\al}{4}}(m_2m_3)^{\frac{2+\al}{4}}. 
    \end{equation}
    Because $\qb_3(t)= (\rb_3^{\frac{2+\al}{2}} \kp t)^{\frac{2}{2+\al}}e^{i\tht^+_3}$ and 
	$$ (\rb_3)^{\frac{2+\al}{2}} \kp = \frac{2+\al}{\sqrt{2\al}}\sqrt{m_2(m_2/m_0)^{1+\al}} = \frac{2+\al}{\sqrt{2\al}} \sqrt{M}, $$
	$\qb_3|_{[0, \infty)}$ is half of a parabolic collision-ejection solution of \eqref{kepler problem}. 

	When $\al \in (1,2)$ (or $\al=1$), for any $|\phi_3 - \tht_3^+| \le \pi$ (or $|\phi_3 - \tht_3^+| < \pi$), by Corollary \ref{coll2}, there is a $C^2$ function $z_3(t) = |z_3(t)| e^{i \tht_3(t)}$, $ t \in [0,T],$ satisfying 
	\begin{equation}
	 z_3 (T) = \qb_3(T), \;\tht_3(0) = \phi_3, \; \tht_3(T) = \tp_3;
	 \end{equation} 
	\begin{equation}
	\label{ineq1} A_{\lb}(z_3, T) < A_{\lb}(\qb_3, T).
	\end{equation}
	Let $z_3|_{[0,T]}$ be the path of $m_3$, and $ z_2(t) = \frac{m_3}{m_2} | z_3(t)| e^{i(\tht_3(t)+\pi)}, t \in [0,T],$ be the path of $m_2$. Then $z(t):= (z_2(t), z_3(t)), t \in [0, T],$ satisfies property $(a)$ and $(b)$ in Lemma \ref{deform22}, and a similar computation as in \eqref{eqnl} shows
		$$ A_{\lk} (z, T) = \frac{m_0 m_3}{m_2} A_{\lk}(z_3, T).$$
	Combining this with \eqref{ineq3} and \eqref{ineq1}, we get
	$$ A_{\lk}(z, T) < A_{\lk}(\qb, T).$$
\end{proof}

\mbox{}

\emph{Acknowledgements.} The author thanks Prof. Richard Montgomery for his valuable suggestion and interest in this work. He also wish to express his gratitude to Prof. Ke Zhang for his continued support and encouragement.

\bibliographystyle{abbrv}
\bibliography{ref-eight}

\end{document}